\pdfoutput=1
\RequirePackage{ifpdf}
\ifpdf
\documentclass[pdftex]{sigma}
\else
\documentclass{sigma}
\fi

\numberwithin{equation}{section}

\usepackage{upgreek}
\usepackage{mathtools}
\usepackage{tikz}
\usepackage[all]{xy}
\usetikzlibrary{matrix,calc,arrows.meta,positioning,patterns}

\newcommand{\bC}{{\mathbb C}}

\newcommand{\bR}{{\mathbb R}}
\newcommand{\bZ}{{\mathbb Z}}
\newcommand{\bN}{{\mathbb N}}

\newcommand{\bT}{{\mathbb T}}

\newcommand{\ev}{\operatorname{ev}}

\DeclareMathOperator{\inj}{inj}

\DeclareMathOperator{\Sec}{sec}

\DeclareMathOperator{\diam}{diam}
\DeclareMathOperator{\Tot}{Tot}
\newtheorem{Theorem}{Theorem}[section]
\newtheorem*{Theorem*}{Theorem}

\newtheorem{tm}[Theorem]{Theorem}
\newtheorem{lm}[Theorem]{Lemma}

\newtheorem{Proposition}[Theorem]{Proposition}

\theoremstyle{definition}

\newtheorem{df}[Theorem]{Definition}
\newtheorem{ques}[Theorem]{Question}

\newtheorem{rem}[Theorem]{Remark}

\begin{document}

\allowdisplaybreaks

\newcommand{\arXivNumber}{2602.08969}

\renewcommand{\thefootnote}{}

\renewcommand{\PaperNumber}{066}

\FirstPageHeading

\ShortArticleName{Symplectic Excision and Distance Rigidity}

\ArticleName{Symplectic Excision and Distance Rigidity\footnote{This paper is a~contribution to the Special Issue on Geometry and Dynamics in memory of Will Merry. The~full collection is available at \href{https://sigma-journal.com/Merry.html}{https://sigma-journal.com/Merry.html}}}

\Author{Yoel GROMAN}
\AuthorNameForHeading{Y.~Groman}
\Address{Department of Mathematics, The Hebrew University of Jerusalem, Jerusalem, Israel}
\Email{\mail{yoel.groman@mail.huji.ac.il}}

\ArticleDates{Received February 09, 2026, in final form June 30, 2026; Published online July 14, 2026}

\Abstract{We consider various notions of completeness in symplectic topology and ask two related questions. Does a~complete open symplectic manifold remain complete after excising a~subset? Can two sets be made arbitrarily far apart by adjusting the almost complex structure within an appropriate class of complete almost complex structures? We find rigidity phenomena when the excised set is a~symplectic hypersurface. These arise from certain open Gromov--Witten invariants. We contrast this with flexibility that often occurs when the excised set is coisotropic. We also briefly touch on the opposite question of obstructions to existence of a~complete symplectic structure compatible with a~given complex structure. For the notion of completeness we first consider the traditional notion of geometric boundedness. We then introduce a~broader notion of normalized completeness, related to the notion of intermittent boundedness of [\textit{Geom. Topol.} \textbf{27} (2023), 1273--1390, arXiv:1510.04265], which depends on $C^0$ properties and is a~contractible condition. Finally, we speculate about the relation to a~Fukaya-categorical notion of completeness.}


\Classification{53D35; 53D40; 53D12}

\begin{flushright}
\emph{Dedicated to the memory of Will J.~Merry}
\end{flushright}

\renewcommand{\thefootnote}{\arabic{footnote}}
\setcounter{footnote}{0}

\section{Introduction}
Let $(M,\omega)$ be an open symplectic manifold and consider the space $\mathcal{J}$ of smooth $\omega$-tame almost complex structures on $M$. Since $\omega$-tameness is a~pointwise condition, the associated metrics~$g_J$ exhibit considerable flexibility. It is natural to ask whether any geometric features are common to all elements of~$\mathcal{J}$. For this question to be productive, one needs to impose additional completeness conditions on the almost complex structures. The geometric features we then consider are distances between subsets of $M$. We find rigidity phenomena when one considers the distance to symplectic hypersurfaces which contrast with flexibility that often occurs when the excised set is coisotropic.

\subsection{The results}
As our first proxy for completeness, we use the well-established notion of \emph{geometric boundedness} (see, e.g., \cite[Section~4.2]{Sikorav}): an~$\omega$-compatible almost complex structure $J$ on $M$ is called \emph{geometrically bounded} if the associated Riemannian metric $g_J(\cdot,\cdot)=\omega(\cdot,J\cdot)$ is complete and has uniformly bounded sectional curvature and uniformly positive injectivity radius.\footnote{To treat $\omega$-tame almost complex structures, one must also impose a~uniform bound on the norm of $\omega$ with respect to $g_J$ (i.e., $|\omega|_{g_J}\le C$).}

We begin with an existence question.
\begin{ques}\label{ques:1}
Given $\omega$, does there exist a~geometrically bounded $\omega$-compatible almost complex structure $J$ on $M$?
\end{ques}
It is customary to call $(M,\omega)$ itself geometrically bounded if the answer to Question~\ref{ques:1} is affirmative. We consider this question when we start with a~geometrically bounded $M$ and excise a~subset $N\subset M$. Question~\ref{ques:1} can be understood as the question: \emph{can the distance to $N$ of any compact subset of $M\setminus N$ be made arbitrarily large by adjusting the almost complex structure?}

\begin{tm}\label{tm:1}
Let $M$ be a~geometrically bounded symplectic manifold with finitely generated spherical homology, and suppose $\omega$ is rational.
Let $N\subset M$ be a~properly embedded symplectic hypersurface.
Then $M\setminus N$ is not geometrically bounded.
\end{tm}

\begin{tm}\label{tm:2}
Let $U\subset \bR^n$ be convex.
Let $\omega$ be the standard symplectic form on $\bR^n\times\bT^n$. If there exists an $\omega$-compatible almost complex structure on $U\times\bT^n$ which is geometrically bounded, then $U=\bR^n$.
\end{tm}

Theorem~\ref{tm:2} follows from the following more quantitative statement which is of independent interest.
Identify primitive classes in $H_1(\bT^n;\bZ)$ with primitive vectors $v\in\bZ^n$, hence with integral conormals.
For an open convex set $U\subset\bR^n$ and $p\in U$, let $L_p := \{p\}\times \bT^n \subset U\times\bT^n$. Define the ``flux distance'' to the boundary by
\[
\delta_{\mathrm{flux}}(p;\partial U):=\inf_{v\in\bZ^n \mathrm{primitive}}\sup_{x\in \partial U}\langle v,p-x\rangle\in(0,+\infty).
\]
\begin{Proposition}\label{prop:flux-distance}
Fix $a>0$. There exists a~universal dimensionless constant $C>0$ such that for every open convex $U\subset\bR^n$, every $p\in U$, and every $\omega$-compatible almost complex structure $J$ on $\bR^n\times\bT^n$ for the standard symplectic form with $g_J$ $a$-bounded $($see Definition~{\rm\ref{df:a-bounded}} below$)$, one has
\[
 d_{g_J}(L_p, (\bR^n\setminus U)\times\bT^n)\leq Ca\delta_{\mathrm{flux}}(p;\partial U)
\]
provided $\delta_{\mathrm{flux}}(p;\partial U)\gg 1/a$.
\end{Proposition}
This is proved in Section~\ref{sec:proofs} and immediately implies Theorem~\ref{tm:2}.
\begin{rem}\label{rem:Normalized} A~more precise formulation is given at the beginning of the proof (equation~\eqref{eqPreciseVersion}). Note that if we use the \emph{normalized distance} $\frac1{a}d_{g_J}$ we obtain an estimate which is \emph{independent of the bound on the geometry}. We revisit this point in Section~\ref{subsec:IntroCompleteness}. It is worth pointing out that since $a$-boundedness implies an injectivity radius lower bound of $1/a$, it naturally constrains the lengths of the circle fibers in $U\times\bT^n$; for a~flat product metric on $T^*\bR^n$, $1/a$ is exactly the length of the $S^1$ fiber.
\end{rem}

\begin{rem}
Theorem~\ref{tm:1} too follows from a~quantitative statement. For certain tori that are close enough to the divisor, we shall show that there is an a~priori upper bound on the distance between the torus and the divisor. This upper bound is linear in a~certain flux distance to the divisor, for small enough flux distance.
\end{rem}

Before proceeding, we point out an elementary obstruction: if $g_J$ is geometrically bounded, G\"unther's volume comparison~\cite{Gallot,Gunther} implies uniform lower bounds on volumes of balls of radius smaller than the injectivity radius.
In particular, \emph{the complement of a~compact subset of a~symplectic manifold is not geometrically bounded.} That is, the theorems above are only non-trivial when the excised set is properly embedded and non-compact. However, \emph{we do not know} whether this is an artifact of this specific notion of completeness. In Section~\ref{subsec:IntroCompleteness}, we revisit the main results considering a~more general notion of completeness. It is not clear whether volume finiteness remains an obstruction. We shall return to this question numerous times in the introduction in connection with $\bC^2\setminus\{0\}$.

We also briefly address the opposite direction.

\begin{ques}
Given $J$, does there exist a~symplectic form $\omega$ on $M$ such that $J$ is $\omega$-compatible and geometrically bounded?
\end{ques}
We prove the following non-existence results.
\begin{tm}\label{tm:3}
Let $M$ be the total space of a~nontrivial effective holomorphic line bundle over a~compact Riemann surface, with its integrable complex structure $J$.
Then $M$ admits no symplectic form $\omega$ such that $J$ is $\omega$-compatible and geometrically bounded.
\end{tm}
\begin{tm}\label{tm:4}
Let $X$ be a~projective variety with an effective divisor $D$ admitting a~curve $C$ such that $C\cdot D>0$.
Then the total space of $\mathcal O(D)$ admits no symplectic form $\omega$ such that the integrable complex structure $J$ is $\omega$-compatible and geometrically bounded.
\end{tm}

\subsection{Holomorphic-curve mechanism and the surface case}

As already indicated, the proofs of Theorems~\ref{tm:1} and~\ref{tm:2} rely on hard holomorphic curve arguments. As is typical for results of this type, the phenomenon is elementary in dimension $2n=2$, and persists in higher dimensions due to holomorphic-curve mechanisms.
Let us expand on this point.

If $p\in D\subset\bC$ is a~point inside a~disk, then $D\setminus\{p\}$ has finite area, so any complete metric on $D\setminus\{p\}$ must have unbounded geometry.
At the same time, there is also a~holomorphic-curve viewpoint: $\partial D$ is a~Lagrangian circle, and $D$ is a~$J$-holomorphic Maslov index $2$ disk with nonvanishing open Gromov--Witten invariant; confinement estimates then yield distance bounds.
This latter mechanism persists in higher dimensions and underlies Theorem~\ref{tm:1}.

Consider now Proposition~\ref{prop:flux-distance} in the $n=1$ case. One might try to vary $J$ so that $g_J$ is stretched in the $\bR$-direction and compressed in the circle direction, in order to make the distance between two circle fibers as large as possible. It is elementary to see that if one insists on a~uniform sectional curvature bound, such stretching forces the circle fibers to become very short somewhere along the way, and hence forces the injectivity radius to become small.

But there is also a~holomorphic-curve mechanism. One can create a~Maslov index $2$ disk with boundary on one of the circle fibers by a~boundary reduction, or, symplectic cutting, at the other end. The resulting open Gromov--Witten invariant again yields a~distance bound. The boundary reduction is why in higher dimensions the result involves the distance between a~Lagrangian fiber and a~real hypersurface.

This suggests the following question.
\begin{ques}
In the setting of Theorem~\ref{tm:2}, fix two distinct torus fibers $L_p,L_q\subset U\times\bT^n$.
Is there an analogue of Proposition~\ref{prop:flux-distance} giving an \emph{a priori} upper bound on $d_{g_J}(L_p,L_q)$ in terms of a~symplectic/flux quantity depending only on the straight-isotopy flux $p-q\in H^1(\bT^n;\bR)$ (equivalently, an integral-affine notion of distance between $p$ and $q$)?
\end{ques}
One is naturally led to consider open Gromov--Witten invariants defined by counting $J$-holomorphic \emph{annuli} with boundary on $L_p\cup L_q$ and with varying modulus.
For primitive boundary classes, one can show (by a~Gromov compactness argument) that an appropriate $0$-dimensional moduli space defines a~well-defined invariant.
In dimension $n=1$, this annulus invariant equals~$1$.
Unfortunately, in higher dimensions the same invariant can be shown to vanish identically. To see this, note first that by the Fukaya trick, the invariant as a~function of $q$, for fixed $p$, is locally constant. On the other hand, in dimension $\ge 2$, one can continuously move $q$ to a~position where the flux from $p$ to $q$ is non-positive. This forces the invariant to vanish by energy considerations.

Finally, since the holomorphic-curve input behind Theorems~\ref{tm:1} and~\ref{tm:2} is inherently codimension $2$, it is natural to ask the following.
\begin{ques}
Do analogues of Theorem~\ref{tm:1} hold for higher-codimension symplectic submani\-folds?
\end{ques}

\subsection{Flexibility of coisotropic excision}\label{sec:flexibility}
The rigidity phenomena above contrast sharply with flexibility results in the \emph{coisotropic} setting. Let us consider two related mechanisms.
\begin{itemize}
\itemsep=0pt
\item \emph{Symplectic excision.}
In~\cite{KarshonTangExcision,StratmannParamRays}, one removes a~closed coisotropic subset $Z\subset (M,\omega)$ (e.g., a~ray, or more generally a~parametrized ray) and constructs a~symplectomorphism \[
\Phi\colon\ M\setminus Z \longrightarrow M.\]
In particular, if $J$ is a~geometrically bounded $\omega$-compatible almost complex structure on~$M$, then $\Phi^*J$ is geometrically bounded on $M\setminus Z$.

\item \emph{Integrable anti-surgery.}
\cite{GromanVarolgunesCompleteEmbeddings} construct a~variety of examples of \emph{complete embeddings}, i.t., excisions $Y:=M\setminus Z$ where $Y$ is \emph{not} symplectomorphic to $M$ but is nevertheless geometrically bounded and embeds symplectically $\iota\colon Y\hookrightarrow M$ with good control at infinity.
These examples arise in particular from singular Lagrangian torus fibrations and an explicit modification procedure called \emph{integrable anti-surgery}.
\end{itemize}

In both cases, flexibility has a~quantitative version: In the symplectic excision case, the symplectomorphism $\Phi$ can be chosen to be the identity outside an arbitrarily small prescribed neighborhood of $Z$. In particular, the distance from a~prescribed compact set to the excised region can be made arbitrarily large while maintaining fixed bounds on the geometry. The case of integrable anti-surgery is similar in this respect. See \cite[Theorem~7.33]{GromanVarolgunesCompleteEmbeddings} for a~precise statement.\looseness=1

Let us emphasize a~key consequence: by excising a~suitably chosen embedded coisotropic subset $Z$, \emph{one can make the distance between \emph{any two prescribed points}---or more generally between a~prescribed point and a~prescribed compact subset---\emph{arbitrarily large} while maintaining fixed bounds on the geometry.}

\subsubsection{A basic 4-dimensional example}
The following example taken from \cite[Remark~7.32]{GromanVarolgunesCompleteEmbeddings} illustrates the contrast between the rigidity and flexibility phenomena. Consider the complex surface
$
X:=\bC^2\setminus\{xy=1\}$,
equipped with the symplectic form
\[
\omega=\operatorname{Im}\frac{{\rm d}x\wedge {\rm d}y}{xy-1}.\]
The manifold $X$ admits a~Lagrangian torus fibration (the Auroux fibration) with a~single focus--focus singular fiber; see \cite[Section~5]{AurouxAnticanonicalTDuality} for a~discussion. It is shown in~\cite{GromanVarolgunesCompleteEmbeddings} that $X$ with this symplectic form is geometrically bounded.

Emanating from the origin (which is a~focus--focus singularity for the Auroux fibration), there are two notable configurations: First, a~pair of properly embedded Lagrangian discs $L^{+}$ and $L^{-}$ (called Lagrangian ``tails'' in~\cite{GromanVarolgunesCompleteEmbeddings}), lying over the monodromy-invariant line in the base of the Auroux fibration. Second, the degenerate cone given by the equation $xy=0$, which consists of a~pair of properly embedded holomorphic copies of $\bC$ intersecting at the origin. This configuration projects to a~properly embedded curve which is transverse to the monodromy-invariant line in the base.

It is shown in~\cite{GromanVarolgunesCompleteEmbeddings} that integrable anti-surgery allows one to \emph{excise either Lagrangian tail while preserving geometric boundedness}; in fact, excising one tail recovers the cotangent bundle $T^*\bT^2$. By contrast, the present paper shows that \emph{excising components of the degenerate cone obstructs geometric boundedness.}

\begin{rem}
In the example above, $X\setminus C_-$ contains the punctured holomorphic curve $C_+$. One might wonder whether the obstruction is explained, forgoing holomorphic curves, simply by the appearance of a~properly embedded symplectic surface whose end has finite area when enclosed by loops.
However this phenomenon alone is not an obstruction: such symplectic surfaces can occur inside geometrically bounded symplectic manifolds.
For example, in the cotangent model $\bR^2\times \bT^2$ with $\omega={\rm d}p_1\wedge {\rm d}q_1+{\rm d}p_2\wedge {\rm d}q_2$, take a~properly embedded curve $\gamma\subset\bR^2$ which is asymptotic to the coordinate axes and has finite total $p_1$-variation along one end; then the surface
\[
S:=\big\{(p,q)\mid p\in\gamma,\, q_2=0,\, q_1\in\bT^1\big\}
\]
is symplectic and has an end of finite symplectic area yet the ambient manifold is geometrically bounded.
\end{rem}
\subsection{Completeness in symplectic topology}\label{subsec:IntroCompleteness}
The obstructions to geometric boundedness in Question~\ref{ques:1} come from $J$-holomorphic curves with finite energy connecting distant compact subsets. This suggests viewing geometric boundedness as a~symplectic counterpart of geodesic completeness~\cite{McDuffSymplecticStructures}, with $J$-holomorphic curves playing the role of geodesics.

However, geometric boundedness has drawbacks: it involves curvature (second derivatives), and the space of geometrically bounded structures may not be contractible or even connected. At the other extreme, one can consider the weakest completeness property.

\begin{df}[$J$-completeness]\label{df:Jcomplete}
An $\omega$-tame almost complex structure $J$ on a~symplectic manifold $(M,\omega)$ is said to be \emph{$J$-complete} if the associated Riemannian metric $g_J$ is complete and for every compact subset $K\subset M$ and every $E>0$ there exists $R>0$ such that any connected $J$-holomorphic curve whose boundary meets both $K$ and $M\setminus B_R(K)$ has energy at least $E$.
\end{df}
This is stronger than metric completeness but too weak to control the topology of the space of $J$-complete structures.

We now consider an intermediate notion between $J$-completeness and geometric boundedness which is of a~$C^0$ nature and which gives rise to a~contractible condition. It is closely related to that of intermittent boundedness which was introduced in~\cite{GromanFloerOpen}. For each $J$, we construct in Section~\ref{sec:norm-complete} an intrinsically defined continuous \emph{scale function} $\rho_J\colon M\to(0,+\infty]$ measuring the isoperimetric scale at each point. The \emph{normalized distance} $d_{\mathrm{norm},J}$ measures path lengths in units of $\rho_J$
\[
L_{\mathrm{norm},J}(\gamma):=\int_{0}^{1}\rho_J(\gamma(t))|\dot\gamma(t)|_{g_J}{\rm d}t,
\]
with $d_{\mathrm{norm},J}$ the induced path metric. The scale function depends on the choice once and for all of a~cutoff constant $c_0>\frac1{4\pi}$. The constant $c_0$ measures the allowable deviation of the metric from being flat whereas the scale function at a~point $p$ measures the size of a~neighborhood in which the metric does not deviate from flatness more than is allowed by the cutoff constant.

\begin{df}[normalized completeness]\label{df:normalized-complete-intro}
An $\omega$-compatible almost complex structure $J$ is \emph{normalized complete} if the metric space $(M,d_{\mathrm{norm},J})$ is complete.
\end{df}
\begin{rem}
In this definition and Theorem~\ref{tm:normalized-completeness}, we focus on the $\omega$-compatible case. The statements are true for the $\omega$-tame case with a~very slight adjustment. See Remark~\ref{rem:omega-tame}.
\end{rem}

\begin{tm}[basic properties of normalized completeness]\label{tm:normalized-completeness}
Let $(M,\omega)$ be a~symplectic mani\-fold.
\begin{enumerate}
\itemsep=0pt
\item[$(1)$] \emph{Connectedness.}
The space of normalized complete $\omega$-compatible almost complex structures is path connected.

\item[$(2)$] \emph{Confinement and $J$-completeness.}
If $J$ is normalized complete, then $J$ is $J$-complete {\rm(}Definition~{\rm\ref{df:Jcomplete})}.
More precisely, any connected $J$-holomorphic curve $u\colon (\Sigma,\partial\Sigma)\to(M,\omega)$ satisfies a~confinement estimate of the form
\[
u(\Sigma)\subset B^{d_{\mathrm{norm},J}}_{CE(u)+C'}(u(\partial\Sigma)),
\]
where \smash{$E(u)=\int_\Sigma u^{*}\omega$} and the constants $C,C'>0$ are universal constants.

\item[$(3)$] \emph{Equivalence to standard metric under uniform bounds.}
If $g_J$ has uniformly bounded sectional curvature and uniformly positive injectivity radius, then $d_{\mathrm{norm},J}$ and $d_{g_J}$ are bi-Lipschitz equivalent. In this case, normalized completeness is equivalent to metric completeness of $g_J$.
\end{enumerate}
\end{tm}
Proofs are given in Section~\ref{sec:norm-complete}.

\begin{rem}[weak contractibility]
The space of normalized complete almost complex structures is path connected, but not necessarily weakly contractible as a~\emph{topological space}. Rather, it can be given the structure of an \emph{$\infty$-groupoid} which is contractible. Namely, one restricts to certain \emph{admissible} paths modeled on the zig--zag construction in the proof of Theorem~\ref{tm:normalized-completeness}\,(1); see~\cite{GromanFloerOpen} for the analogous framework with uniform intermittent boundedness.
\end{rem}

\begin{rem}
By considering, the set of almost complex structures $J$ that are normalized complete for all cutoff constants $c_0 > \frac{1}{4\pi}$, we obtain a~contractible class of almost complex structures whose definition involves no arbitrary choices.
By standard comparison theorems, this contractible class contains all geometrically bounded metrics, since curvature precisely controls the scale at which deviation from flatness occurs above any given threshold.
For the purposes of Floer theory, one generally uses perturbations that decay at infinity, and this class is closed under such perturbations.
Another choice-free contractible class consists of all $J$ whose associated metric is uniformly equivalent to one in the previous class. Any such $J$ is easily seen to be normalized complete with respect to some cutoff constant, providing a~notion that is stable under uniform equivalence.
\end{rem}

\begin{rem}
 The previous remark allows us to define a~robust class of normalized complete structures with good properties and no arbitrary choices.
 However, to define an actual metric~$d_{\mathrm{norm},J}$, one must commit to a~specific choice of cutoff constant and scale function. For a~given~$J$, the resulting normalization can depend very strongly on this choice, yielding distances that are not necessarily uniformly equivalent for different cutoffs.
 The exception is when~$J$ is already geometrically bounded: in that case, the resulting normalized metrics are always bi-Lipschitz equivalent to the standard metric $g_J$ (see Theorem~\ref{tm:normalized-completeness}\,(3)).
 \end{rem}

\begin{tm}[reformulation with normalized completeness]\label{tm:norm-complete-reformulation}
The main results of this work hold with geometric boundedness replaced by normalized completeness. Concretely:
\begin{itemize}\itemsep=0pt
\item {\bf Theorem~\ref{tm:1}.} If $(M,\omega)$ is normalized complete {\rm(}with the same hypotheses on~$M$ and~$\omega$ as in Theorem~{\rm\ref{tm:1})} and $N\subset M$ is a~properly embedded symplectic hypersurface, then $M\setminus N$ is not normalized complete.
\item {\bf Theorem~\ref{tm:2}.} Let $U\subset\bR^n$ be convex. If there exists an $\omega$-compatible almost complex structure on $U\times\bT^n$ {\rm(}with the standard symplectic form{\rm)} that is normalized complete, then~${U=\bR^n}$.
\item {\bf Proposition~\ref{prop:flux-distance}.} For a~universal constant $C$, we have
\[
d_{\mathrm{norm},J}(L_p,(\bR^n\setminus U)\times\bT^n) <C\delta_{\mathrm{flux}}(p;\partial U)
\]
provided $\delta_{\mathrm{flux}}(p;\partial U)\gg 1$.
\item {\bf Theorem~\ref{tm:3}.} The total space of a~nontrivial effective holomorphic line bundle over a~compact Riemann surface, with its integrable complex structure $J$, admits no symplectic form $\omega$ such that $J$ is $\omega$-compatible and normalized complete.
\item {\bf Theorem~\ref{tm:4}.} Let $X$ be a~projective variety with an effective divisor $D$ admitting a~curve~$C$ such that $C\cdot D>0$. Then the total space of $\mathcal{O}(D)$ admits no symplectic form $\omega$ such that the integrable complex structure $J$ is $\omega$-compatible and normalized complete.
\end{itemize}
\end{tm}

We conclude with the following question.

\begin{ques}
 Does $\bC^2\setminus\{0\}$ admit a~normalized complete $\omega$-compatible almost complex structure?
 \end{ques}

\subsection{Fukaya-categorical completeness}\label{subsec:fukaya-complete}
So far we have considered completeness from a~geometric point of view.
One can also consider completeness in terms of moduli spaces of objects in the Fukaya category. That is, completeness in the sense that there are no \emph{missing objects}. To discuss this we recall the definition of the flux of a~Lagrangian isotopy and the star-shape.

Let $I\colon L\times[0,1]\to M$ be a~smooth isotopy through Lagrangian embeddings, and write $\iota_s:=I(\cdot,s)$.
For each $s$, let $X_s$ be the vector field along $\iota_s(L)$ defined by $X_s\circ \iota_s=\partial_s I$.
Because~$\iota_s(L)$ is Lagrangian, the $1$-form $\lambda_s:=\iota_s^*(\iota_{X_s}\omega)$ on $L$ is closed.
The \emph{(symplectic) flux} of the isotopy from $s=0$ to $s=1$ is the cohomology class
\[
\mathrm{Flux}(I) := \left[\int_0^1 \lambda_s{\rm d}s\right] \in H^1(L;\bR).
\]
 A~\emph{star-isotopy} is a~Lagrangian isotopy whose flux develops linearly in time. The \emph{star-shape} ${\rm Sh}^*(L) \subset H^1(L; \bR)$ of a~Lagrangian $L$ is the set of all fluxes of star-isotopies starting at $L$ (see~\cite{EntovGanorMembrez,STVFluxShape}). It can be thought of as a~tropicalization of a~moduli space of objects of the Fukaya category which are connected to~$L$.
The following proposition is proved at the end of Section~\ref{subsec:Proof:tm:1}.
\begin{Proposition}[{cf.\ \cite[Theorem~2.2]{STVFluxShape}}]\label{prop:starshape-hypersurface}
Let $(M,\omega)$ be a~geometrically bounded symplectic manifold.
Let $N$ be a~properly embedded symplectic hypersurface.

For each $p\in N$, and a~sufficiently small open neighborhood $U \subset M$ of $p$, there exists a~Lagrangian $L\subset U$, and a~class $\alpha\in\pi_2(M,L)$ with $\alpha\cdot N=1$, such that the following holds.
Let ${\rm Sh}^*(L; M\setminus N)\subset H^1(L;\bR)$ be the star-shape of $L$ relative to $M\setminus N$. Then the star-shape~${{\rm Sh}^*(L; M\setminus N)}$ is strictly contained within the open affine half-space
$
\bigl\{\mathfrak f\in H^1(L;\bR) \mid \mathfrak f\cdot\partial\alpha>-\omega(\alpha)\bigr\}$.
Furthermore, the boundary of ${\rm Sh}^*(L; M\setminus N)$ intersects the affine bounding hyperplane~${\{\mathfrak f\cdot\partial\alpha=-\omega(\alpha)\}}$ in a~nontrivial open subset. Points in this boundary subset arise as limits of star-fluxes realized by Lagrangian tori embedded entirely within $M\setminus N$.

\end{Proposition}
This shows an additional defect that occurs when one removes a~symplectic divisor: some boundary fluxes in the star-shape are ``missing'', in the sense that they cannot be realized either by a~Lagrangian or by a~$J$-holomorphic disc. The disc can be seen as encoding some compactification of the moduli space. By contrast, the proof of geometric boundedness of integrable anti-surgery in~\cite{GromanVarolgunesCompleteEmbeddings} proceeds, in effect, by showing that some relevant star-shapes are all of $H^1(L;\bR)$.

This suggests that the geometric notion of normalized completeness is closely linked with some notion of completeness of moduli spaces of objects in the Fukaya category.
\begin{ques}\label{ques:fukaya-boundary}
In a~normalized-complete symplectic manifold $(M,\omega)$, is it always the case that every boundary segment of the star-shape ${\rm Sh}^*(L)$ (for a~given Lagrangian $L$) is accounted for by some $J$-holomorphic disc?
\end{ques}

\section{Proofs}\label{sec:proofs}

\subsection[Energy confinement for J-holomorphic curves]{Energy confinement for $\boldsymbol{J}$-holomorphic curves}
Throughout, if $u\colon (\Sigma,\partial\Sigma)\to (M,\omega)$ is a~$J$-holomorphic curve with (possibly empty) boundary, its energy is
$
E(u):=\int_{\Sigma} u^{*}\omega \ge 0$.

The following confinement lemmas are stated for $\omega$-compatible $J$. At the end of this subsection we note that both hold for $\omega$-tame $J$ under the uniform tameness condition from the footnote to the definition of geometric boundedness on page \pageref{tm:1} (in Section 1.1).

\begin{lm}[Sikorav's confinement estimate]\label{lem:sikorav}
 Let $(M,\omega)$ be a~symplectic manifold and let $J$ be an $\omega$-compatible almost complex structure such that the associated metric $g_J(\cdot,\cdot):=\omega(\cdot,J\cdot)$ is complete and has uniformly bounded sectional curvature and positive injectivity radius $($i.e., $J$ is geometrically bounded$)$.
 Then there exist a~constant $r_0>0$ depending on the geometry and a~universal constant $c>0$ such that for every connected $J$-holomorphic curve $u$ one has
 \[
 u(\Sigma)\subset B_{r_0\,\lceil E(u)/(c r_0^2)\rceil}(u(\partial\Sigma)),
 \]
 where $B_{r}(\cdot)$ denotes the closed $g_J$-ball of radius $r$ about a~subset and $\lceil\cdot\rceil$ is the ceiling function.
 \end{lm}
 \begin{proof}
 See \cite[Proposition 4.4.1]{Sikorav}.
 \end{proof}
 This confinement estimate is concerned with the large scale. Namely, it contains no information about curves with energy much smaller than $cr_0^2$ other than that the diameter is $\leq r_0$. In some cases, e.g., in the proof of Proposition~\ref{prop:starshape-hypersurface}, we need an estimate for the small scale. This is taken care of by the following lemma.
 \begin{lm}[improved confinement]\label{lem:sikorav-improved}
 In the setting of Lemma~{\rm\ref{lem:sikorav}}, with the same $r_0$ and $c$, and letting $C := 1/c$, one has
 \[
 u(\Sigma)\subset B_{\rho}(u(\partial\Sigma)),\qquad \rho:=\max\bigl(3CE(u)/r_0,\sqrt{CE(u)}\bigr).
 \]
 \end{lm}
 \begin{proof}
 The argument relies on the same monotonicity inequality used for Sikorav's lemma: if a~$J$-holomorphic curve meets the center of a~ball of radius $r\le r_0$ and the boundary of the $J$-holomorphic curve does not meet the interior of the ball, then the energy in that ball is at least $c r^2$.

 If $E(u)/\bigl(c r_0^2\bigr)\le 1$, we apply the monotonicity inequality directly: if the image contained a~point at the center of a~ball of radius greater than \smash{$\sqrt{E(u)/c}$} whose interior did not meet $u(\partial\Sigma)$, then monotonicity would force $E(u)\ge c r^2>E(u)$ (since $E(u)\ge c r^2$ implies $r\le\sqrt{E(u)/c}$). Hence $u(\Sigma)$ is contained in a~ball of radius \smash{$\sqrt{E(u)/c} = \sqrt{C E(u)}$} about $u(\partial\Sigma)$.

 If $E(u)/\bigl(c r_0^2\bigr)>1$, take the maximum number $n$ of disjoint balls of radius $r_0$ that the image of $u$ can meet. Then $n\le E(u)/\bigl(c r_0^2\bigr)$ by monotonicity, and the image is contained in a~ball of radius $(2n+1)r_0\le r_0\bigl(1+2E(u)/\bigl(c r_0^2\bigr)\bigr)<3E(u)/(cr_0) = 3CE(u)/r_0$ around $u(\partial\Sigma)$.
 \end{proof}

For later use, we also record a~localized version, originally observed in~\cite{GromanFloerOpen}.
\begin{df}\label{df:a-bounded}
 For $a>0$, we say that the complete Riemannian metric $g_J$ is \emph{$a$-bounded at a~point} $p\in M$ if \smash{$\inj_{g_J}(p)\ge \frac{1}{a}$} and $\Sec_{g_J}\le a^{2}$ on $B_{1/a}(p)$. We say that $g_J$ is \emph{$a$-bounded on a~set} if it is $a$-bounded at every point of that set.
\end{df}
\begin{rem}
 When $g_J$ is $a$-bounded, we can take $r_0=1/(2a)$ \cite[Theorem 4.4.1 and Remark~4.4.3]{McDuffSalamonJHol}.
\end{rem}

\begin{lm}[local confinement]\label{lem:local-confinement}
Fix $a>0$. There exist $r_0=r_0(a)>0$ and a~universal constant $C>0$ with the following property. Let $K\subset M$ be compact and let $R>0$.
Assume~$J$ is $\omega$-compatible and $g_J$ is complete and is $a$-bounded on the $g_J$-ball $B_{R}(K)$.
Let $u$ be a~$J$-holomorphic curve $u$ with $u(\partial\Sigma)\subset K$ and write
\[
\rho:=\max\bigl(2CE(u)/r_0,\sqrt{CE(u)}\bigr).
\]
If $\rho\le R$, then $u(\Sigma)\subset B_R(K)$.
\end{lm}
\begin{proof}
The proof of Lemma~\ref{lem:sikorav-improved} relies on the monotonicity inequality which depends only on the local geometry of a~complete Riemannian metric. So taking $r_0$ and $C$ as in Lemma~\ref{lem:sikorav-improved}, we obtain by the same argument that
$
u(\Sigma)\subset B_{\rho}(u(\partial\Sigma))
$.
\end{proof}

\begin{rem}
Lemmas~\ref{lem:sikorav}, \ref{lem:sikorav-improved}, and~\ref{lem:local-confinement} also hold for $\omega$-tame almost complex structures $J$ that are geometrically bounded (respectively $a$-bounded on the relevant set) and satisfy the uniform tameness condition $|\omega|_{g_J}\le C$. The constants then depend on the tameness bound as well. See Remark~\ref{rem:omega-tame}.
\end{rem}

\subsection{Proof of Theorem~\ref{tm:1}}\label{subsec:Proof:tm:1}
We now turn to the proof of Theorem~\ref{tm:1}.
The idea is to pick a~small Lagrangian torus $L$ linking the hypersurface~$N$, consider the class of the obvious ``meridian'' holomorphic disc intersecting~$N$ once, and show that the corresponding moduli space has nonzero evaluation degree.

Fix $p\in N$. Choose a~Darboux chart $\psi\colon U\to \bC^{n}$ centered at $p$ such that $\psi(p)=0$ and, in these coordinates,
$
\psi(U\cap N)=\bigl(\{0\}\times \bC^{n-1}\bigr)\cap B_{r}(0)
$
for some $r>0$.
For sufficiently small $t>0$, define a~Lagrangian torus $L_t\subset U$ by
\[
L_t:=\{(z_1,\dots,z_n)\in \bC^n \mid |z_1|=t, |z_i|=10t \text{ for }i=2,\dots,n\}.
\]
Let $\alpha\in \pi_2(M,L_t)$ be the relative homotopy class represented in the chart by the disc
$
u_0\colon  (D,\partial D)\allowbreak\to (\bC^n,L_t)$, $ u_0(\zeta)=(t\zeta,0,\dots,0)$,
so $E(u_0)=\int_D u_0^*\omega_{\mathrm{std}}=\pi t^{2}$.

For an almost complex structure $J$ on $M$, denote by $\mathcal{M}_{J}(L_t,\alpha)$ the moduli space of $J$\nobreakdash-holo\-morphic discs $u\colon(D,\partial D)\to(M,L_t)$ in class $\alpha$ with one boundary marked point, modulo reparametrization. Let
$
\ev\colon \mathcal{M}_{J}(L_t,\alpha)\to L_t
$
be the evaluation map at the marked boundary point.

\begin{lm}\label{lem:degree-one}
There exist arbitrarily small $t > 0$ such that the following holds.
For a~generic geometrically bounded $\omega$-tame almost complex structure $J$ on $M$, the moduli space $\mathcal{M}_{J}(L_t,\alpha)$ is a~smooth closed manifold of dimension $n$, and the evaluation map $\ev\colon\mathcal{M}_{J}(L_t,\alpha)\to L_t$ has degree $1$.
\end{lm}
\begin{proof}
Step 1: \emph{Index and regularity.}
Trivialize $TM$ over the chart and use the standard computation for the product torus to see that the Maslov index of $\alpha$ equals $2$.
Since the disc is simple (in particular not multiply covered), for a~generic choice of $\omega$-tame $J$ the linearized $\bar\partial_J$-operator is surjective (see, e.g., \cite{McDuffSalamonJHol}), hence the regular part $\mathcal{M}_{J,reg}(L_t,\alpha)$ consisting of maps from the disc to $M$ is a~smooth manifold of the expected dimension $n+\mu(\alpha)-2=n$.

Step 2: \emph{Compactness $($no bubbling$)$.}
Since $\omega$ is a~rational form and $H_2(M; \mathbb{Z})$ is finitely generated, the image of $\omega$ on $\pi_2(M)$ is a~discrete subgroup of $\mathbb{R}$, say $\kappa\mathbb{Z}$ for some $\kappa > 0$. Thus there exists $\hbar > 0$ such that every nonconstant $J$-holomorphic sphere has area at least~$\hbar$. Choose a~sufficiently large integer $m$ and define $t > 0$ by the relation $\pi t^2 = \kappa / m$. By choosing~$m$ large enough, we ensure both $\pi t^2 < \hbar$ and $t < r/10$, so the construction stays well inside the Darboux chart. The image of $\omega$ on the relative homotopy group $\pi_2(M, L_t)$ is generated over $\mathbb{Z}$ by the absolute periods on $\pi_2(M)$ and the areas of the basic discs spanning the factors of the torus~$L_t$ in the chart. These generators are $\kappa$, $\pi t^2 = \kappa / m$, and $100\pi t^2 = 100\kappa / m$. Consequently, the area of any relative class in $\pi_2(M, L_t)$ is an integer multiple of $\kappa / m$.
 Any stable map in class $\alpha$ has total area $\pi t^2 = \kappa / m$. It cannot contain a~nonconstant sphere component because~${\kappa / m < \hbar}$. Moreover, any disc bubbling would split $\alpha$ as a~sum of two or more nontrivial relative classes, each requiring a~strictly positive area. Since the minimal positive area available in $\pi_2(M, L_t)$ is exactly~${\kappa / m}$, and the total area of $\alpha$ is $\kappa / m$, such a~split is impossible. Thus~${\mathcal{M}_{J,\text{reg}}(L_t, \alpha) = \mathcal{M}_J(L_t, \alpha)}$ is a~closed manifold.

Step 3: \emph{Independence of $\deg(\ev)$ and normalization.}
Let $J_0$, $J_1$ be geometrically bounded and $\omega$-tame.
By \cite[Theorem 4.15\,(a)]{GromanFloerOpen} (or, alternatively, Theorem~\ref{tm:normalized-completeness}\,(2),~(3)), one can choose a~generic path $\{J_s\}_{s\in[0,1]}$ from $J_0$ to $J_1$ so the moduli space of pairs $\{(s,u)\mid u\in \mathcal{M}_{J_s}(L_t,\alpha)\}$ is contained in a~compact set. Since disc bubbling is excluded as above, this gives a~compact cobordism showing that $\deg(\ev)$ is independent of $J$.

Finally, take $J$ which agrees with the standard complex structure on the Darboux chart and is geometrically bounded on the complement of a~compact set (hence geometrically bounded globally). For $t$ sufficiently small, Lemma~\ref{lem:sikorav} forces any $J$-holomorphic disc in class $\alpha$ with boundary on $L_t$ to remain inside the chart, where the computation reduces to the standard toric case. In that case, regularity and $\deg(\ev)=1$ are well known; see~\cite{ChoOhToric,FOOO_Toric1}.
\end{proof}

We can now prove Theorem~\ref{tm:1}.
\begin{proof}[Proof of Theorem~\ref{tm:1}]
Suppose by contradiction that $M\setminus N$ admits a~geometrically bounded $\omega$-tame almost complex structure $J$, and choose $a>0$ such that $g_J$ is $a$-bounded.
Fix the small parameter $t$ so that Lemma~\ref{lem:degree-one} applies, and note that $\alpha$ has algebraic intersection number $1$ with $N$ (it is the meridian disc in the local model).

Set $E:=E(u_0)=\pi t^2$ and \smash{$R>\max\bigl(2CE/r_0(a),\sqrt{CE}\bigr)$}, with $r_0(a)$ and $C$ from Lemma~\ref{lem:local-confinement} (so that the confinement radius \smash{$\rho=\max\bigl(2CE/r_0,\sqrt{CE}\bigr)\le R$}).
Let $B:=B_{R}(L_t)\subset M\setminus N$ be the closed $g_J$-neighborhood of $L_t$ of radius $R$.
Since $J$ is geometrically bounded on $M\setminus N$, the metric $g_J$ is complete; hence closed bounded subsets are compact, so $B$ is compact.

Choose an $\omega$-tame almost complex structure $J'$ on $M$ which agrees with $J$ on an open neighborhood of $B$ and is geometrically bounded on $M$ (e.g., interpolate to a~fixed geometrically bounded structure outside a~compact set).
By Lemma~\ref{lem:degree-one}, there exists a~stable $J'$-holomorphic disc $u$ in class $\alpha$ with boundary on $L_t$.
Its energy equals $E$.

Since $J'=J$ near $B$ and $g_J$ is $a$-bounded on $B$, Lemma~\ref{lem:local-confinement} implies $u(D)\subset B$. (We use local confinement here because $J'$ is only guaranteed to be $a$-bounded on $B$, not globally on $M$.)
But~${\alpha\cdot N=1}$ forces any representative of $\alpha$ to intersect $N$, contradicting $B\subset M\setminus N$.
\end{proof}

\subsection{Proof of Theorem~\ref{tm:2}}

Let $X$ be a~closed symplectic toric manifold with moment polytope $P\subset \bR^n$ and moment map~${\mu\colon X\to P}$.
For $p\in \operatorname{Int} P$, let $L_p:=\mu^{-1}(p)\cong \bT^n$. A~\emph{basic $($Maslov index $2)$ class} in~$\pi_2(X,L_p)$ is a~relative class $\beta_F$ corresponding to a~facet $F$ of $P$.
Geometrically, $\beta_F$ can be constructed as follows. Let $v_F \in \mathbb{Z}^n$ be the primitive integral inward-pointing normal vector to~$F$. The circle action generated by $v_F$ acts freely on the interior of $P$ but collapses to a~point over the facet~$F$. Thus, if we draw a~line segment $L$ in the moment polytope $P$ from $p$ to a~point in the interior of $F$ and, transporting along the orbit of this action as one traverses $L$, we obtain a~disc in the class $\beta_F$.

\begin{Proposition}\label{prop:toric-disc-Psi}
For any $p\in \operatorname{Int} P$ and any $\omega$-tame almost complex structure $J$ on $X$, there exists a~stable $J$-holomorphic map with boundary on $L_p$ whose domain has at least one disc component $u_D\colon (D,\partial D)\to (X,L_p)$ of positive symplectic area such that
\[
\int_D u^*\omega \le \min_{F}\omega(\beta_F),
\]
where the minimum is taken over facets $F$ of $P$.
\end{Proposition}
\begin{rem}
 The inequality above is stated for the \emph{disc component} $u_D$: The configuration guaranteed by the proposition may contain sphere bubbles. In that case, the disc component has strictly smaller area than $\omega(\beta_F)$.
 \end{rem}
\begin{proof}
We use the $\Psi$-invariant of a~Lagrangian submanifold considered by Shelukhin--Tonko\-nog--Vianna~\cite{STVFluxShape}.
Fix an $\omega$-tame almost complex structure $J$ and auxiliary choices (virtual perturbations) defining the Fukaya $A_\infty$-algebra of $L_p$, and pass to a~classically minimal model.
The quantity $\Psi(L_p)\in(0,+\infty]$ is defined as an infimum of valuations of the Maurer--Cartan prepotential $m\bigl(e^{b}\bigr)$; it is independent of all choices (in particular of $J$) by~\cite[Theorem~4.2]{STVFluxShape}.
Furthermore, expanding the definition, $\Psi(L_p)$ is the minimum of the symplectic areas $\omega(\beta)$ over those relative classes $\beta\in H_{2}(X,L_p;\bZ)$ which contribute nontrivially to some symmetrized $A_\infty$ operation; see~\cite[Theorem~4.8]{STVFluxShape}. Note this includes the classes which contribute to the curvature term.

For any such class $\beta$ contributing to the curvature term there exists a~stable $J$-holomorphic map $u$ representing $\beta$ with at least one nonconstant disc component.\footnote{To see this, note that otherwise $\beta$ would be a~purely spherical class. The invariance of the indicator function~\cite[Proposition~5.1]{STVFluxShape} would then dictate that it must also realize the $\psi$ invariant for $J_{\mathrm{toric}}$.
However, for $J_{\mathrm{toric}}$, all non-constant holomorphic spheres are confined to the toric boundary divisor and are strictly disjoint from the interior fiber $L_p$, so their contributions to the $A_\infty$-algebra of $L_p$ necessarily vanish.
Hence, $\partial \beta \neq 0$.}
The class area decomposes as a~sum of the symplectic areas of the disc and sphere components of $u$, so every nonconstant disc component $u_D$ satisfies
\[
\int_D u_D^*\omega \le \omega(\beta),
\]
with strict inequality whenever $u$ also has a~nonconstant sphere component.

Now take the standard toric almost complex structure $J_{\mathrm{toric}}$.
Fukaya--Oh--Ohta--Ono~\cite{FOOO_Toric1} (and earlier Cho--Oh~\cite{ChoOhToric}) compute the leading-order terms of the toric superpotential/potential function: for each facet $F$ the basic Maslov index $2$ class $\beta_F$ contributes nontrivially to the curvature term (equivalently, to the potential), with coefficient $1$; see \cite[Section~11]{FOOO_Toric1} and~\cite{ChoOhToric}.
Therefore, for a~facet $F$ minimizing $\omega(\beta_F)$, the class $\beta_F$ witnesses $\Psi(L_p)$ for $J_{\mathrm{toric}}$, and hence
\[
\Psi(L_p) \le \omega(\beta_F) = \min_{F'}\omega(\beta_{F'}).
\]
Since $\Psi(L_p)$ is independent of $J$, the same inequality holds for our given $J$.
Applying \cite[Theorem~4.8]{STVFluxShape} for this $J$, we obtain a~stable $J$-holomorphic map representing some class $\beta$ with~${\omega(\beta)=\Psi(L_p)\le \min_{F'}\omega(\beta_{F'})}$.
Choosing a~disc component $u_D$ of this stable map, the discussion above gives
\[
\int_D u_D^*\omega\le \omega(\beta)=\Psi(L_p)\le \min_{F'}\omega(\beta_{F'}),
\]
which is the claimed inequality.
\end{proof}

To prove Theorem~\ref{tm:2}, it suffices to prove Proposition~\ref{prop:flux-distance} from the introduction.

Write $\pi\colon \bR^{n}\times\bT^{n}\to \bR^{n}$ for the projection, and for $q\in\bR^n$ write $L_q:=\{q\}\times\bT^n\subset \bR^n\times\bT^n$.
Recall the definition of $\delta_{\mathrm{flux}}(p;\partial U)$ from the introduction.
For later use, given $v\in\bZ^n$ we write
\[
c_{\partial U}(v):=\inf_{x\in \partial U}\langle v,x\rangle\in[-\infty,+\infty),
\qquad
\delta_{\partial U}(p;v):=\langle v,p\rangle-c_{\partial U}(v)\in(0,+\infty],
\]
so $\delta_{\partial U}(p;v)<+\infty$ precisely when $c_{\partial U}(v)>-\infty$.

We will deduce Theorem~\ref{tm:2} and Proposition~\ref{prop:flux-distance} from the following stronger version of Proposition~\ref{prop:flux-distance}.

\begin{Proposition}\label{prop:flux-distance-strong}
Let $U\subset\bR^n$ be an open convex set, and let $V$ be an open neighborhood of~${U\times\bT^n}$ in $\bR^n\times\bT^n$. Let $J$ be an $\omega$-tame almost complex structure on $V$ such that the associated metric $g_J$ is complete on $V$ and is $a$-bounded on $U\times\bT^n$ for some $a>0$. Then for every $p\in U$,
\[
 d_{g_J}(L_p, V\setminus (U\times\bT^n)) \le\max\bigl(4Ca\delta_{\mathrm{flux}}(p;\partial U),\sqrt{C\delta_{\mathrm{flux}}(p;\partial U)}\bigr),
\]
where $C$ is a~universal constant.
\end{Proposition}

To prove Proposition~\ref{prop:flux-distance-strong}, we formulate a~basic lemma.
 \begin{lm}\label{lemma2.11}
 Let $U\subset\bR^n$ be an open convex set, let ${K\subset U}$ be compact, and let ${v\in\bZ^n}$ be primitive with $c_{\partial U}(v)>-\infty$.
 Then there exists a~compact Delzant polytope $P$ such that $K\subset \operatorname{Int} P\subset U$ and such that one facet of $P$ has outward primitive normal $v$ {\rm(}equivalently, is parallel to the supporting hyperplane $\{\langle v,\cdot\rangle=c_{\partial U}(v)\}${\rm)}.
 \end{lm}
\begin{proof}
Choose $\varepsilon>0$ such that the closed $\varepsilon$-neighborhood $K_\varepsilon:=\{x\mid d(x,K)\le \varepsilon\}$ is still contained in $U$ (possible since $U$ is open and $K$ is compact).
Pick $c<c_{\partial U}(v)$ close enough to~$c_{\partial U}(v)$ so that $K_\varepsilon\subset\{x\mid \langle v,x\rangle>c\}$, and set $H:=\{x\mid \langle v,x\rangle=c\}$ and $H^-:=\{x\mid \langle v,x\rangle\ge c\}$.

Since $K_\varepsilon$ is compact, we can find finitely many supporting half-spaces with \emph{rational} (hence integral after scaling) outward normals whose intersection contains $K_\varepsilon$ and is contained in $U$. Intersecting these half-spaces together with $H^{-}$ yields a~compact simple rational polytope $P_0$ with $K\subset \operatorname{Int} P_0\subset U$ and with a~facet parallel to $H$.

Finally, by applying arbitrarily small \emph{corner cuts} one can make the polytope smooth/Delzant while keeping the containments $K\subset \operatorname{Int} P\subset U$ and without changing the facet with outward normal $v$.
\end{proof}
\begin{proof}[Proof of Proposition~\ref{prop:flux-distance-strong}]
Let $\delta:=\delta_{\mathrm{flux}}(p;\partial U)$.
Fix $\varepsilon>0$, and choose a~primitive $v\in\bZ^n$ with $c_{\partial U}(v)>-\infty$ and $\delta_{\partial U}(p;v)\le \delta+\varepsilon$.
Put $c:=c_{\partial U}(v)$ and $\ell(x):=\langle v,x\rangle-c$, so $\ell(p)=\delta_{\partial U}(p;v)$.

Let $J$ be as in the statement. We will prove the inequality
\begin{equation}\label{eqPreciseVersion}
 d_{g_J}(L_p, V\setminus (U\times\bT^n)) \le \max\bigl(4Ca\delta_{\mathrm{flux}}(p;\partial U),\sqrt{C\delta_{\mathrm{flux}}(p;\partial U)}\bigr).
\end{equation}
Let $R>0$ be such that the closed metric ball $B:=B_R(L_p)$ with respect to $g_J$ is contained in~${U\times\bT^n}$.
Since $g_J$ is complete on $V$ and $B\subset U\times\bT^n\subset V$, the closed ball $B$ is compact. In particular, $\pi(B)\subset U$ is compact.

Apply Lemma~\ref{lemma2.11} to $K:=\pi(B)\subset U$ and the chosen $v$ to obtain a~Delzant polytope~$P$ with
$
\pi(B)\subset \operatorname{Int} P\subset U
$
and with a~facet $F\subset \partial P$ whose outward primitive normal is~$v$.
Let~${(X_P,\omega_P,\mu_P)}$ be the associated compact symplectic toric manifold with moment polytope~$P$.
By the action--angle identification, $\mu_P^{-1}(\operatorname{Int}P)$ is symplectomorphic to $\operatorname{Int}P\times\bT^n$; we use this to view $B$ as a~subset of $X_P$.
Extend $J|_{B}$ arbitrarily to an $\omega_P$-tame almost complex structure~$\overline{J}$ on $X_P$.

Let $A:=\min_{F'}\omega_P(\beta_{F'})$, where $F'$ ranges over facets of $P$.
Since $F$ is one of the facets, we have $A\le \omega_P(\beta_F)$.
Moreover, because $P\subset U$ and the facet $F$ has outward normal $v$, the defining constant $c_P$ of that facet satisfies $c_P\ge c_{\partial U}(v)=c$, hence
\[
\omega_P(\beta_F)=\langle v,p\rangle-c_P \le \langle v,p\rangle-c = \delta_{\partial U}(p;v) \le \delta+\varepsilon.
\]
By Proposition~\ref{prop:toric-disc-Psi} applied to $(X_P,\omega_P,L_p)$ and $\overline{J}$, there exists a~stable $\overline{J}$-holomorphic map with a~disc component $u_D$ and $\partial u_D\subset L_p$ satisfying
\[
E(u_D)=\int_D u_D^*\omega_P \le A \le \delta+\varepsilon.
\]
The disc component $u_D$ must meet the toric boundary
$
X_P\setminus \mu_P^{-1}(\operatorname{Int}P)=\mu_P^{-1}(\partial P)$.
Indeed, one can pick a~primitive of $\omega_P$ on $\mu_P^{-1}(\operatorname{Int}P)$ for which $L_p$ is exact, showing that any disc with boundary on $L_p$ with positive symplectic area must meet the toric boundary.

On the other hand, since $g_{\overline{J}}$ is $a$-bounded on $B=B_R(L_p)$, Lemma~\ref{lem:local-confinement} implies that $u_D(D)$ is contained in the ball of radius \smash{$\rho=\max\bigl(2CE(u_D)/r_0,\sqrt{CE(u_D)}\bigr)$} about $L_p$, with $r_0=r_0(a)$ and the universal constant $C$ from that lemma. Since $u_D$ meets $\mu_P^{-1}(\partial P)$ while, by construction, $B_R(L_p)$ does not, we must have $R\le\rho\le C(a)\max\bigl(\delta+\varepsilon,\sqrt{\delta+\varepsilon}\bigr)$.
Since $\varepsilon>0$ was arbitrary, this gives the claim.
\end{proof}

\begin{proof}[Proof of Proposition~\ref{prop:flux-distance}]
Apply Proposition~\ref{prop:flux-distance-strong} with $V = \bR^n\times\bT^n$. Since $J$ is $\omega$-compatible and $g_J$ is $a$-bounded globally on $\bR^n\times\bT^n$, it is $a$-bounded on $U\times\bT^n$. Moreover, $g_J$ is complete on $V$. The conclusion of Proposition~\ref{prop:flux-distance-strong} exactly yields the desired inequality.
\end{proof}

\begin{proof}[Proof of Theorem~\ref{tm:2}]
Suppose for contradiction that $U\neq \bR^n$. Since $U$ is open and convex, its boundary is nonempty, so $\delta_{\mathrm{flux}}(p;\partial U)$ is finite for any $p\in U$.
The hypothesis states that there exists a~geometrically bounded $\omega$-compatible almost complex structure $J$ on $U\times\bT^n$. This means $g_J$ is complete on $U\times\bT^n$ and is $a$-bounded for some $a>0$.
Apply Proposition~\ref{prop:flux-distance-strong} with~${V = U\times\bT^n}$. Since $V\setminus (U\times\bT^n) = \varnothing$, the distance $d_{g_J}(L_p, V\setminus (U\times\bT^n))$ is infinite. However, Proposition~\ref{prop:flux-distance-strong} provides a~finite upper bound in terms of $\delta_{\mathrm{flux}}(p;\partial U)$, which is a~contradiction. Hence $U=\bR^n$.
\end{proof}

\subsection{Proof of Proposition~\ref{prop:starshape-hypersurface}}

\begin{proof}[Proof of Proposition~\ref{prop:starshape-hypersurface}]
 Fix $p\in N$. Choose a~Darboux chart $\psi\colon U\to\bC^n$ centered at $p$ with $\psi(p)=0$ and $\psi(U\cap N)=\bigl(\{0\}\times\bC^{n-1}\bigr)\cap B_r(0)$ for some $r>0$; in the chart, $N\cap U$ is given by $\{z_1=0\}$.
 For $t>0$ sufficiently small, let $L_{\max}\subset U$ be the Lagrangian product torus~${|z_1|=t}$, $|z_i|=10t$ for $i\ge 2$.

 Pick a~geometrically bounded almost complex structure $J$ which coincides with the standard one in the Darboux chart. If a~Lagrangian $L$ is contained in a~smaller pre-compact open set $V$ such that $\overline{V}\subset U$, which we fix henceforth, Lemma~\ref{lem:sikorav-improved} guarantees there is an $\hbar>0$ so that any $J$-holomorphic disc with boundary on a~Lagrangian contained in $V$ of energy less than $\hbar$ stays confined in $U$.

 Choose $t>0$ so small that $100\pi t^2<\hbar$.
 Identify $H^1(L_{\max};\bR)\cong\bR^n$ using the basis given by the boundary loops of the coordinate discs for the product torus.
 With respect to this basis, consider the rectangle
 \[
 R := \bigl(-\pi t^2, 0\bigr] \times \bigl[-19\pi t^2, 0\bigr]^{n-1} \subset H^1(L_{\max};\bR).
 \]
 Figure~\ref{fig:flux-rectangle} illustrates the case $n=2$ in flux coordinates $H^1(L_{\max};\bR)\cong\bR^2$: the torus $L_{\max}$ is the origin, and the rectangle $R$ has $L_{\max}$ as its upper-right corner, extending down and left precisely until it hits the bounding hyperplane.

 \begin{figure}[ht]
 \centering
 \begin{tikzpicture}[>=Latex, thick, font=\small, scale=0.9]
 \tikzset{yaxis/.style={red!70!black, line width=1.6pt, thin}}

 \draw[->] (-2.2,0) -- (4.5,0) node[below] {$x$};
 \draw[yaxis, ->] (0,-0.5) -- (0,6.5);
 \node[anchor=south] at (0.12,6.55) {$y$};
 \node[anchor=north west, red!70!black] at (-2.9,0.55) {hypersurface $N$};

 \coordinate (Lmax) at (1.5, 5.5);
 \coordinate (Rbl) at (0, 2.55);

 \fill[fill=orange!28, fill opacity=0.45, draw=orange!80!black, thick] (Rbl) rectangle (Lmax);
 \node[orange!80!black, align=center] at (2.75,4.02) {$R$};

 \draw[yaxis] (0,-0.5) -- (0,6.5);

 \fill (Lmax) circle (2pt) node[above right] {$L_{\max}$};

 \coordinate (Lstar) at (0.75, 4.02);
 \fill[blue] (Lstar) circle (2pt) node[below right, text=blue] {$L^*$};

 \draw[dashed, thin, gray] (Lmax) -- (0, 5.5) node[midway, above] {$\pi t^2$};
 \draw[dashed, thin, gray] (Lmax) -- (1.5, 0) node[midway, right] {$100\pi t^2$};

 \draw[<->, thin, orange!80!black] (-0.2, 2.55) -- (-0.2, 5.5) node[midway, left] {$19\pi t^2$};

 \end{tikzpicture}
 \caption{Schematic representation of the flux rectangle. The $y$-axis is the projection of the hypersurface~$N$. The bounding point $L_{\max}$ sits at the top right. The chosen base Lagrangian $L^*$ lies strictly in the interior of the realizable region $R$, providing space in all directions before hitting the bounds.}
 \label{fig:flux-rectangle}
 \end{figure}

 \emph{Realizability inside the Darboux chart.}
 Inside $U\setminus N$, we have the explicit family of product Lagrangian tori
 \[
 K_{r_1,r_2,\dots,r_n}
 :=\{(z_1,\dots,z_n)\in U\setminus N \mid |z_1|=r_1,\, |z_i|=r_i \text{ for }i\ge 2\},
 \]
 with $r_1\in(0,t]$ and $r_i\in(9t,10t]$ for $i\ge 2$ (so $K_{t,10t,\dots,10t}=L_{\max}$).
 Each such torus is contained in the Darboux chart and determines a~flux $\mathfrak f(r_1,\dots,r_n)\in H^1(L_{\max};\bR)$.
 Varying the radii exactly realizes the region $R$.

 To satisfy the claim of the proposition, we explicitly pick the base Lagrangian $L^*$ to be a~torus from this family whose flux lies strictly in the interior of $R$. For instance, let $L^* := K_{t/2, 9.5t, \dots, 9.5t}$. Let $\alpha^* \in \pi_2(M, L^*)$ be the corresponding meridian class, which has area $\omega(\alpha^*) = \pi (t/2)^2$.

 \emph{Obstruction for arbitrary star-isotopies.}
 Now let $\{K_s\}_{s\in[0,1]}$ be an \emph{arbitrary} Lagrangian \emph{star}-isotopy in $M\setminus N$ starting at $K_0=L^*$, and let $\mathfrak f\in H^1(L^*;\bR)$ be its total flux.
 Because $K_s$ undergoes a~straight isotopy, the symplectic area of any relative homology class changes affinely:
 \[
 \omega(\alpha^*_s) = \omega(\alpha^*) +s(\mathfrak{f}\cdot\partial\alpha^*).
 \]
  Consider the $\Psi$-invariant $\Psi(s) := \Psi(K_s)$.

 We first prove $\Psi(s) = \omega(\alpha^*_s)$ on a~small interval $[0,\delta]$.
 At $s=0$, the Lagrangian $L^*$ lies inside the Darboux chart. The meridian class $\alpha^*$ is a~regular Maslov index $2$ disc with area~${\omega(\alpha^*) = \pi (t/2)^2}$. Any other $J$-holomorphic disc on $L^*$ either leaves $U$ \big(energy $\ge \hbar > 100\pi t^2$\big) or stays in $U$ and is a~standard disc of energy at least $\pi(9t)^2 > 80\pi t^2$.

 To extend this analysis to small $s>0$, we apply Fukaya's trick. See \cite[Section~2.1]{STVFluxShape}.
 Namely, we choose diffeomorphisms $f_s\colon M\to M$ with $f_s(L^*)=K_s$, and define pushed-forward almost complex structures $J_s:=(f_s)_*J$.
 Composition with $f_s^{-1}$ gives a~bijection between $J_s$-holomorphic discs with boundary on $K_s$ and $J$-holomorphic discs with boundary on $L^*$. Under the star-isotopy with total flux $\mathfrak f$, the area of any continued relative class $\beta_s\in\pi_2(M,K_s)$ changes by exactly the flux
$
 \omega(\beta_s)=\omega(\beta)+s(\mathfrak f\cdot\partial\beta)$.
 Since $f_s$ depends smoothly on $s$ and $f_0=\mathrm{id}$, the structures $J_s$ remain $\omega$-tame and the exit energy threshold degrades continuously. Thus, for a~small enough interval $[0,\delta]$, it still takes at least $100\pi t^2$ energy for any $J_s$-holomorphic disc on~$K_s$ to exit the Darboux chart $U$. For the discs that stay inside $U$, they correspond to the known basic classes. Their areas $\omega(\beta_s)$ differ from their initial area (which is at least $80\pi t^2$) by the flux variation $s(\mathfrak f\cdot\partial\beta)$, which is $O(\delta)$.
 While these areas may decrease for sufficiently small $\delta$, the initial area gap between the meridian $\alpha^*$ and all other discs is vast, so the areas of all other discs cannot decrease by enough to drop below the area $\omega(\alpha^*_s) = \pi (t/2)^2 + s(\mathfrak f\cdot\partial\alpha^*)$ of the meridian class.
 Therefore, $\alpha^*_s$ remains the unique minimal-area holomorphic disc class contributing to the potential, and we conclude that $\Psi(s)=\omega(\alpha^*_s)$ for all $s\in[0,\delta]$.

 It is shown in~\cite{STVFluxShape} that $\Psi(s) > 0$ for all $s$ and that $\Psi(s)$ is strictly concave on $[0,1]$. Since a~concave function that coincides with a~linear function on an initial segment must be bounded above by that linear function everywhere, we have $\Psi(s) \le \omega(\alpha^*_s)$ for all $s\in [0,1]$.
 Since~${\omega(\alpha^*_s) \ge \Psi(s) > 0}$, evaluating at $s=1$ yields $\omega(\alpha^*) + \mathfrak f\cdot\partial\alpha^* > 0$, so
$
 \mathfrak f\cdot\partial\alpha^*>-\omega(\alpha^*)$,
 which proves that the star-shape of $L^*$ is contained in the open half-space.

 Finally, consider the boundary of the star-shape of $L^*$. In the flux space $H^1(L^*; \bR)$, the affine hyperplane $\{\mathfrak{f} \cdot \partial\alpha^* = -\omega(\alpha^*)\}$ corresponds geometrically to the limit as $r_1 \to 0$. The closure of the explicitly realized region $R$, translated to the base point $L^*$, intersects this bounding hyperplane exactly in the face corresponding to $r_1 = 0$ while $r_i \in [9t, 10t]$. Because this face has positive extent in the remaining dimensions, it constitutes a~nontrivial open subset of the hyperplane. Since every flux in $R$ is realized by a~product torus in $M\setminus N$, this open subset is approached by realizable star-fluxes, completing the proof.
 \end{proof}

\subsection{Proof of Theorems~\ref{tm:3} and~\ref{tm:4}}
\begin{proof}[Proof of Theorem~\ref{tm:3}]
Let $L\to \Sigma$ be a~nontrivial effective holomorphic line bundle over a~compact Riemann surface, and let $M:=\Tot(L)$ with its integrable complex structure $J$.
Suppose for contradiction that there exists a~symplectic form $\omega$ on $M$ such that $J$ is $\omega$-compatible and geometrically bounded.
Let $g_J$ be the associated metric.

Since $L$ is effective, it admits a~nonzero holomorphic section $\sigma$.
Fiberwise complex scalar multiplication defines a~holomorphic $\bC^{*}$-action on $\Tot(L)$; in particular, for each $t>0$ the map
\[
m_t\colon\ \Tot(L)\to \Tot(L),\qquad m_t(\ell):=t\ell
\]
is biholomorphic.
Thus the maps $u_t\colon \Sigma\to M$ given by $u_t:=m_t\circ \sigma$ are $J$-holomorphic.
Their images $u_t(\Sigma)\subset M$ are compact complex curves. The curves $u_t$ are all cohomologous since they are connected by an isotopy, so the symplectic area of $u_t$ is independent of $t$. By the improved confinement estimate (see Lemma~\ref{lem:sikorav-improved}), we have $\diam_{g_J}(u_t(\Sigma))\le C\max\bigl(E(u_0)/r_0,\sqrt{E(u_0)}\bigr)$ for a~constant $C$ and $r_0$ depending on the geometry.
On the other hand, since $L$ is nontrivial, any holomorphic section has a~zero, so choose $x_1\in \Sigma$ with $\sigma(x_1)=0$; then $u_t(x_1)$ lies in the zero section for all $t$. It follows $u_t(\Sigma)$ is contained in a~fixed ball around the zero section. By completeness, this ball is compact. But for any $x_0\in \Sigma$ with $\sigma(x_0)\neq 0$ we have $u_t(x_0)$ escaping every compact subset of $M$ as $t\to\infty$. Contradiction.
\end{proof}

\begin{proof}[Proof of Theorem~\ref{tm:4}]
Let $X$ be a~projective variety, let $D\subset X$ be an effective divisor, and let $M:=\Tot(\mathcal{O}(D))$ with its integrable complex structure $J$.
Suppose for contradiction that there exists a~symplectic form $\omega$ on $M$ such that $J$ is $\omega$-compatible and geometrically bounded.

Let $s_D$ be the canonical holomorphic section of $\mathcal{O}(D)$ whose zero locus is $D$.
Let $C\subset X$ be a~curve not contained in $D$ with $C\cdot D>0$. Then for any $t\neq 0$ the curve $s_D(C)$ is a~curve in $M$ which is not contained in the zero section. As in the previous proof, fiberwise scalar multiplication $m_t$ on $\Tot(\mathcal{O}(D))$ is biholomorphic, hence the curves $C_t:=m_t(s_D(C))$ are $J$-holomorphic with symplectic area independent of $t$. They pass through the zero section for all $t$ (at the zeros of $s_D(C)$), while also escaping every compact subset as $t\to\infty$ (at a~point where $s_D(C)\neq 0$). Therefore $\diam_{g_J}(C_t)\to\infty$. This contradicts the confinement estimate (see Lemma~\ref{lem:sikorav-improved}).
\end{proof}

\section{Normalized completeness}\label{sec:norm-complete}
\subsection{An isoperimetric scale}\label{subsec:isoperimetric-scale}
Let $(M,\omega)$ be a~symplectic manifold and let $J$ be an $\omega$-compatible almost complex structure.
Write
$
g_J(\cdot,\cdot):=\omega(\cdot,J\cdot)
$
for the associated Riemannian metric, and write $\ell_{g_J}(\gamma)$ for the $g_J$-length of a~piecewise smooth loop $\gamma\subset M$.

\begin{df}[isoperimetric ratio on a~ball]\label{df:isoperimetric-ratio}
For $p\in M$ and $r>0$, define
\[
\mathcal{I}_J(p;r)
:=
\sup_{\gamma}
\inf_{u}
\frac{\int_D u^*\omega}{\ell_{g_J}(\gamma)^2} \in[0,+\infty],
\]
where the supremum ranges over all \emph{contractible} piecewise smooth loops $\gamma\subset B^{g_J}_r(p)$ with $\ell_{g_J}(\gamma)>0$, and the infimum ranges over all \emph{smooth} discs $u\colon (D,\partial D)\to \bigl(B^{g_J}_r(p),\gamma\bigr)$ with~${\partial u=\gamma}$.
\end{df}

Fix once and for all a~cutoff constant $c_0>0$.

\begin{df}\label{df:isoperimetric-scale}
For $p\in M$, let
\[
r_{\mathrm{iso},J}(p)
:=
\sup\{r\in[0,+\infty) \mid \mathcal{I}_J(p;r)\le c_0\}
\]
and define the \emph{isoperimetric scale}
$
\rho_J(p):=\min\{1,r_{\mathrm{iso},J}(p)\}\in[0,1]$.
\end{df}
\begin{rem}\label{rem:cutoff}
The constant $c_0$ is an arbitrary fixed numerical cutoff. The effect of changing $c_0$ is discussed below.
\end{rem}

\begin{lm}[properties of the scale function]\label{lem:rho-continuous}
The function $\rho_J\colon M\to[0,1]$ is positive everywhere and $1$-Lipschitz with respect to $d_{g_J}$.
In particular, $\rho_J$ is bounded away from $0$ on compact subsets.
\end{lm}
\begin{proof}
 Positivity follows from the fact that inside a~small enough geodesic ball one has curvature bounds, hence a~quadratic isoperimetric inequality for filling contractible loops by discs; see \cite[Theorem~4.4.1]{McDuffSalamonJHol}.

Fix $p,q\in M$ and set $\delta:=d_{g_J}(p,q)$. Then for every $r>\delta$ one has the inclusion $B^{g_J}_{r-\delta}(q)\subset B^{g_J}_{r}(p)$. By Definition~\ref{df:isoperimetric-ratio}, this implies
 \[
 \mathcal{I}_J(q;r-\delta) \le \mathcal{I}_J(p;r).
 \]
 Consequently, if $\mathcal{I}_J(p;r)\le c_0$, then $\mathcal{I}_J(q;r-\delta)\le c_0$, hence
 $
 r_{\mathrm{iso},J}(q)\ge r_{\mathrm{iso},J}(p)-\delta.
 $
 By symmetry we also obtain $r_{\mathrm{iso},J}(p)\ge r_{\mathrm{iso},J}(q)-\delta$, so
 $
 |r_{\mathrm{iso},J}(p)-r_{\mathrm{iso},J}(q)|\le d_{g_J}(p,q).
 $
 This proves the claim, and hence the same for $\rho_J=\min\{1,r_{\mathrm{iso},J}\}$.
 \end{proof}

\subsection{Normalized length and normalized completeness}\label{subsec:normalized-distance}
\begin{df}[normalized length and distance]\label{df:normalized-distance}
Let $J$ be $\omega$-compatible and let $\rho_J$ be as in Definition~\ref{df:isoperimetric-scale}.
For a~piecewise smooth path $\gamma\colon [0,1]\to M$, define its \emph{normalized length} by
\[
L_{\mathrm{norm},J}(\gamma):=\int_{0}^{1}\rho_J(\gamma(t)) |\dot\gamma(t)|_{g_J}{\rm d}t.
\]
The induced path metric is denoted $d_{\mathrm{norm},J}$.
Equivalently, $d_{\mathrm{norm},J}$ is the distance associated to the conformal metric \smash{$\widetilde g_J:=\rho_J^{2}g_J$}.
\end{df}
\begin{rem}
Note that the metric $\widetilde g_J$ is generally not smooth.
\end{rem}
\begin{df}[normalized completeness]\label{df:normalized-complete}
An $\omega$-compatible almost complex structure $J$ is \emph{normalized complete} if the metric space $(M,d_{\mathrm{norm},J})$ is complete.
\end{df}

\begin{lm}[closed bounded balls are compact]\label{lem:bounded-balls-compact}
 If $J$ is normalized complete, then for any basepoint $p_0\in M$ and any $n\in\bN$, the closed $d_{\mathrm{norm},J}$-ball \smash{$K_n:=\overline B^{d_{\mathrm{norm},J}}_{n}(p_0)$} is compact in the manifold topology.
 \end{lm}
 \begin{proof}
 Since $\rho_J$ is continuous and positive, we can approximate it by a~smooth positive function~$\widetilde\rho_J$ such that the ratio $\rho_J/\widetilde\rho_J$ is uniformly bounded away from $0$ and~$+\infty$.
 The smooth Riemannian metric $\widehat g_J:=\widetilde\rho_J^{2}g_J$ is then uniformly equivalent to $\widetilde g_J=\rho_J^{2}g_J$, hence the distances~$d_{\widehat g_J}$ and~$d_{\mathrm{norm},J}$ are bi-Lipschitz equivalent.
 Since $(M,d_{\mathrm{norm},J})$ is complete, $(M,d_{\widehat g_J})$ is also complete.
 Each closed ball \smash{$K_n=\overline B^{d_{\mathrm{norm},J}}_{n}(p_0)$} is bounded with respect to $d_{\widehat g_J}$ (by uniform equivalence), and since $\widehat g_J$ is smooth, the Hopf--Rinow theorem implies that closed bounded sets are compact.
 Thus each $K_n$ is compact in the manifold topology.
 \end{proof}

\subsection{Proof of Theorem~\ref{tm:normalized-completeness}}

\subsubsection[Connectedness (Theorem~1.6 (1))]{Connectedness (Theorem~\ref{tm:normalized-completeness}\,(1))}
\begin{proof}
The proof of connectedness is precisely the zigzag construction in \cite[Theorem~4.7]{GromanFloerOpen}.

Denote by $\mathcal{J}^{\mathrm{comp}}_\omega$ the space of smooth $\omega$-compatible almost complex structures with the $C^\infty$ topology and by \smash{$\mathcal{J}^{\mathrm{nc,comp}}_\omega\subset \mathcal{J}^{\mathrm{comp}}_\omega$} the subspace of normalized complete structures.
Assume \smash{$\mathcal{J}^{\mathrm{nc,comp}}_\omega\neq\varnothing$}.
Without loss of generality, assume that $M$ is connected and fix a~basepoint~${p_0\in M}$.
Given $J_0,J_1\in\mathcal{J}^{\mathrm{nc,comp}}_\omega$, consider the closed normalized balls
\[
B_n^0:=\overline B^{d_{\mathrm{norm},J_0}}_{n}(p_0),\qquad B_n^1:=\overline B^{d_{\mathrm{norm},J_1}}_{n}(p_0),\qquad n\in\bN.
\]
By Lemma~\ref{lem:bounded-balls-compact}, each of these balls is compact.
By construction of the normalized distance, successive balls in each sequence are at normalized distance at least $1$
\begin{equation}\label{eq:shell-width}
d_{\mathrm{norm},J_0}\bigl( B_n^0, M\setminus B_{n+1}^0\bigr)\ge 1,\qquad d_{\mathrm{norm},J_1}\bigl( B_n^1, M\setminus B_{n+1}^1\bigr)\ge 1,\qquad n\in\bN.
\end{equation}

We choose inductively increasing subsequences $n_1^0<n_2^0<\cdots$ and $n_1^1<n_2^1<\cdots$ in $\bN$ with the following nesting property: for each $i\ge 1$,
\[
B_{n_i^0+1}^0 \subset \operatorname{Int} B_{n_i^1}^1
\qquad \text{and} \quad B_{n_i^1+1}^1 \subset \operatorname{Int}B_{n_{i+1}^0}^0.
\]
(So the ball of radius $n_i^0+1$ for $J_0$ is contained in the ball of radius $n_i^1$ for $J_1$, and the ball of radius $n_i^1+1$ for $J_1$ is contained in the ball of radius \smash{$n_{i+1}^0$} for $J_0$.)
Such subsequences exist because both sequences exhaust $M$ as $n\to\infty$.
For each $i\ge 1$, set \smash{$V_i^0:=\operatorname{Int}B_{n_{i}^0+1}^0\setminus B_{n_i^0}^0$} and \smash{$V_i^1:=\operatorname{Int}B_{n_{i}^1+1}^1\setminus B_{n_i^1}^1$}.
These are pairwise disjoint shells escaping to infinity, and each has normalized transverse width at least $1$ with respect to the corresponding $J_0$ or $J_1$ by~\eqref{eq:shell-width}.

We choose a~smooth path $J_s$, $s\in[0,1]$, in $\mathcal{J}^{\mathrm{comp}}_\omega$ from $J_0$ to $J_1$ (through $\omega$-compatible structures) which satisfies
\[
J_s\equiv J_0\quad \text{on }\bigcup_i V_i^0\ \text{for }s\in[0,2/3],
\qquad
J_s\equiv J_1\quad \text{on }\bigcup_i V_i^1\ \text{for }s\in[1/3,1].
\]

We claim that each $J_s$ is normalized complete.
Indeed, fix $s\in[0,1]$. Then at least one of the two collections of shells,
$\big\{V_i^0\big\}$ (if $s\le 2/3$) or $\big\{V_i^1\big\}$ (if $s\ge 1/3$),
consists of shells on which~$J_s$ agrees with a~normalized complete structure (either $J_0$ or $J_1$) and hence has normalized transverse width $\ge 1$.
Any $d_{\mathrm{norm},J_s}$-Cauchy sequence cannot cross infinitely many such shells (each crossing costs at least $1$), so it is eventually contained in some \smash{$B_{n_i^0}^0$} or \smash{$B_{n_i^1}^1$}, which is compact in the manifold topology.
Since $d_{\mathrm{norm},J_s}$ induces the manifold topology (it is the length metric of a~continuous conformal factor $\rho_{J_s}^{2}g_{J_s}$ with $\rho_{J_s}>0$), Cauchy sequences in a~compact set converge.
Thus $(M,d_{\mathrm{norm},J_s})$ is complete.
This shows in particular that $\mathcal{J}^{\mathrm{nc,comp}}_\omega$ is path connected.
\end{proof}

\subsubsection[Normalized confinement and J-completeness (Theorem 1.16 (2))]{Normalized confinement and $\boldsymbol{J}$-completeness (see Theorem~\ref{tm:normalized-completeness}\,(2))}
\begin{proof}

 The proof is an adjustment to Sikorav's confinement argument \cite[Proposition~4.4.1]{Sikorav}, with the bounded-geometry radius replaced pointwise by the scale function $\rho_J$ coming from the local quadratic isoperimetric inequality $\mathcal{I}_J\le c_0$.

 Fix $p\in M$ and set $r:=\rho_J(p)\in(0,1]$.
 By definition of $r_{\mathrm{iso},J}$ and $c_0$, every contractible loop~${\gamma\subset B^{g_J}_{r}(p)}$ has a~smooth filling $u$ with $\int_D u^*\omega\le c_0\,\ell_{g_J}(\gamma)^2$.
 The standard monotonicity estimate for $J$-holomorphic curves (see \cite[Proposition 4.3.1]{Sikorav}) therefore applies on the ball~$B^{g_J}_{r}(p)$ and gives: if $u$ is a~$J$-holomorphic curve whose image meets $p$ and whose boundary is disjoint from~$B^{g_J}_{r}(p)$, then
 \begin{equation}\label{eq:mono-scale}
 E\bigl(u; u^{-1}\bigl(B^{g_J}_{r}(p)\bigr)\bigr) \ge \frac{r^2}{4c_0}.
 \end{equation}

 Let $u\colon(\Sigma,\partial\Sigma)\to(M,\omega)$ be a~connected $J$-holomorphic curve, and set $A:=u(\partial\Sigma)$.
 We define the core of the image as the set of points sufficiently far from the boundary:
 \[ U := \{ x \in u(\Sigma) \mid d_{g_J}(x, A) \ge \rho_J(x) \}. \]
 For every $x \in U$, the extrinsic ball \smash{$B^{g_J}_{\rho_J(x)}(x)$} is disjoint from $A$. Consequently, applying the monotonicity estimate~\eqref{eq:mono-scale} to each such ball guarantees that the energy of the curve within it is at least $\rho_J(x)^2 / 4c_0$.

 Consider the family of extrinsic balls \smash{$\mathcal{B} = \big\{ B^{g_J}_{\rho_J(x)}(x) \big\}_{x \in U}$}. By the Vitali covering lemma, there exists a~countable subcollection
 \[
 \big\{ B_i = B^{g_J}_{a_i}(x_i) \big\}_{i \in I} , \qquad \text{with $a_i = \rho_J(x_i)$},
 \]
consisting of pairwise disjoint balls such that their $5$-times expansions \smash{$\widehat{B}_i = B^{g_J}_{5a_i}(x_i)$} completely cover $U$.
 Because the balls $B_i$ are pairwise disjoint in $M$, their preimages are disjoint in $\Sigma$. Summing the energy over this subcollection avoids any overcounting, yielding
 \[
 \sum_{i \in I} a_i^2 \le 4c_0 \sum_{i \in I} E\bigl(u; u^{-1}(B_i)\bigr) \le 4c_0 E(u).
 \]

 We now estimate the $d_{\mathrm{norm},J}$-diameter of the expanded balls $\widehat{B}_i$. For any two points $y, z \in \widehat{B}_i$, consider a~$g_J$-geodesic $\gamma$ connecting them; its length satisfies $L \le 10a_i$. Since $\rho_J$ is $1$-Lipschitz, the value of the scale function along $\gamma$ is bounded by $\rho_J(x_i) + d_{g_J}(x_i, \gamma(t)) \le a_i + 5a_i = 6a_i$. The normalized distance is thus bounded by the integral of the scale along $\gamma$
 \[
 d_{\mathrm{norm},J}(y, z) \le \int_\gamma \rho_J(\gamma(t)) {\rm d}t \le (6a_i) \cdot L \le 60 a_i^2.
 \]
 Thus, the normalized diameter of each $\widehat{B}_i$ is at most $60a_i^2$.
 Because the normalized metric integrates the scale, the normalized diameter of \smash{$\widehat{B}_i$} is bounded by its maximum scale ($6a_i$) times its $g_J$-diameter ($10a_i$), which is at most $60a_i^2$.

 For any point $p \in u(\Sigma)$, we can now bound its normalized distance to the boundary $A$.
 If~${p \notin U}$, then $d_{g_J}(p, A) < \rho_J(p)$. A~$g_J$-shortest path from $p$ to $A$ has $g_J$-length at most $\rho_J(p)$, and the scale along this path is at most $2\rho_J(p)$. Thus, its normalized length is at most $2\rho_J(p)^2 \le 2$.
 If $p \in U$, any path in $M$ connecting $p$ to $A$ must exit $U$. We can trace a~path to the boundary of $U$ entirely within the cover $\big\{\widehat{B}_i\big\}$. The normalized distance to traverse this cover is strictly bounded by the sum of the normalized diameters of all the expanded balls. Once outside $U$, the remaining normalized distance to $A$ is at most $2$.

 We can therefore unconditionally estimate the normalized distance from any $p \in u(\Sigma)$ to the boundary $A$ by combining these bounds
 \[
 d_{\mathrm{norm},J}(p,A) \le \sum_{i \in I} 60a_i^2 + 2 \le 240 c_0 E(u) + 2.\tag*{\qed}
 \] \renewcommand{\qed}{}
\end{proof}

In the formulation of Theorem~\ref{tm:normalized-completeness}\,(2), we only dealt with the larger energy regime. Here we record a~localized version which includes also the low-energy regime.
\begin{lm}[normalized local confinement]\label{lem:normalized-local-confinement}
Let $J$ be an $\omega$-compatible almost complex structure that is normalized complete.
There exist universal constants $C_1,C_2>0$ such that for every connected $J$-holomorphic curve $u\colon(\Sigma,\partial\Sigma)\to(M,\omega)$,
\[
u(\Sigma) \subset B^{d_{\mathrm{norm},J}}_{\rho}(u(\partial\Sigma)),\qquad \rho := \max\bigl(C_1\,E(u),C_2\,\sqrt{E(u)}\bigr).
\]
In particular, as $E(u)\to 0$ the $d_{\mathrm{norm},J}$-diameter of $u(\Sigma)$ tends to zero. Moreover, this containment is unaffected by the behavior of $J$ outside the normalized ball of radius $\rho$.
\end{lm}
\begin{proof}
We revisit the end of the proof of Theorem~\ref{tm:normalized-completeness}\,(2). Using the same notation, the normalized distance from any point $p \in u(\Sigma)$ to the boundary $A = u(\partial\Sigma)$ is at most the sum of the normalized diameters of the expanded balls $\widehat{B}_i$ covering the core $U$, plus the normalized distance from an exit point $q \notin U$ to $A$.

As observed there, the normalized diameter of the cover $\bigcup \widehat{B}_i$ is bounded by $\sum_i 60a_i^2$. By monotonicity, $\sum_i a_i^2 \le 4c_0 E(u)$, so the sum of the diameters is at most $240c_0 E(u)$.

For the final segment from $q \notin U$ to $A$, the $g_J$-distance is some length $L < \rho_J(q) \le 1$. Because the scale function is globally bounded by $1$, the normalized distance of this shortest path is at most its $g_J$-length $L$. Furthermore, the standard monotonicity estimate applied to the ball $B^{g_J}_L(q)$ gives $E(u) \ge L^2 / 4c_0$. Thus, the normalized distance from $q$ to $A$ is bounded by~$\smash{L \le \sqrt{4c_0 E(u)}}$.

Combining these yields
\[
d_{\mathrm{norm}, J}(p, A) \le 240c_0 E(u) + \sqrt{4c_0 E(u)}.
\]

If $E(u) \ge 1$, we have $\sqrt{E(u)} \le E(u)$, and we can absorb the square root term into the linear term by taking $C_1 = 240c_0 + \sqrt{4c_0}$.
If $E(u) < 1$, we have $E(u) < \sqrt{E(u)}$, and we can absorb the linear term to get an estimate $\le C_2 \sqrt{E(u)}$ by taking $C_2 = 240c_0 + \sqrt{4c_0}$.

The last sentence is immediate since the monotonicity estimate depends only on the local behavior.
\end{proof}

\subsubsection[Equivalence to standard metric under uniform bounds (Theorem 1.16 (3))]{Equivalence to standard metric under uniform bounds (Theorem~\ref{tm:normalized-completeness}\,(3))}
\begin{proof}
Under the assumption,
$g_J$ is $a$-bounded at every point $p\in M$. See Definition~\ref{df:a-bounded}. So the scale function satisfies $\rho_J(p)\in [2/a,1]$. See \cite[Lemma~4.10]{GromanFloerOpen}. Therefore, the scaled metric is bi-Lipschitz equivalent to the metric $g_J$.
\end{proof}

\begin{rem}[$\omega$-tame case]\label{rem:omega-tame}
If instead of $\omega$-compatible one uses an $\omega$-tame almost complex structure that is geometrically bounded and satisfies the uniform tameness condition $|\omega|_{g_J}\le C$, then one still obtains an isoperimetric scale and scale function as in Section~\ref{subsec:isoperimetric-scale}, but with a~cutoff constant $c_0'$ depending on $C$ in place of $c_0=\tfrac12$. See \cite[Section~4.4]{McDuffSalamonJHol} and \cite[Section~3, Appendix~B]{GromanWrappedSemiToricSYZ}. The statements of Theorem~\ref{tm:normalized-completeness}\,(2) and (3) then hold with the same proofs after adjusting constants.
\end{rem}

\subsection{Proof of Theorem~\ref{tm:norm-complete-reformulation}}
\begin{proof}
In all cases, we replace the local confinement estimate (see Lemma~\ref{lem:local-confinement}) by the normalized local confinement estimate (see Lemma~\ref{lem:normalized-local-confinement}) and the proofs then are word for word the same.
\end{proof}

\subsection*{Acknowledgements}
I am grateful to Umut Varolgunes and Jake P.\ Solomon for helpful discussions. I would like to thank the referees for their careful reading and for their comments. They have helped weed out numerous mistakes and have improved the paper greatly.
This work was supported by ISF grant 3605/24.

I used AI-assisted tools for editing and improving the exposition (in particular for wording, structure, and catching typographical issues).
All mathematical content, results, and any remaining errors are my own responsibility.

\pdfbookmark[1]{References}{ref}
\LastPageEnding

\end{document}